\documentclass[12pt, a4paper]{amsart}

\usepackage{amsmath,amssymb,amsthm, url}
\usepackage[all]{xy}
\usepackage{enumerate, fullpage}
\usepackage[dvipdfmx]{graphicx}
\usepackage{wrapfig}
\usepackage{mathtools}
\usepackage{ascmac}
\usepackage[foot]{amsaddr}

\usepackage{mathrsfs}
\usepackage[mathscr]{euscript}

\theoremstyle{plain}
\numberwithin{equation}{section}
\newtheorem{thm}{Theorem}[section]
\newtheorem{prop}[thm]{Proposition}
\newtheorem{cor}[thm]{Corollary}
\newtheorem{lem}[thm]{Lemma}

\newtheorem{prob}[thm]{Problem}

\newtheorem{question}{Question}

\theoremstyle{definition}
\newtheorem{dfn}[thm]{Definition}

\newtheorem{rmk}[thm]{Remark}

\def\rank{\mathop{\mathrm{rank}}\nolimits}
\def\dim{\mathop{\mathrm{dim}}\nolimits}

\def\Hom{\mathop{\mathrm{Hom}}\nolimits}

\def\<{{\langle}}
\def\>{{\rangle}}

\def\+{\mathop{\oplus}\nolimits}

\def\Supp{\mathop{\mathrm{Supp}}\nolimits}

\def\1{\mathop{\mathrm{id}}\nolimits}
\def\Spec{\mathop{\mathrm{Spec}}\nolimits}

\newcommand{\gl}[2]{{\mathsf{gl}\left({#1},  {#2}\right)}}%\def\gl{\mathop{\mathsf{gl}}\nolimits}

\newcommand{\HNF}[3]{{
\xymatrix{
0	\ar[r]	&	{#2}_1	\ar[r]\ar[d]	&	{#2}_2	\ar[r]\ar[d]	&	\cdots\ar[r]	&	{#2}_{n-1}	\ar[r]\ar[d]	&	{#2}_n={#1}\ar[d]	\\
			&	{#3}_1\ar@{-->}[ul]	&	{#3}_2\ar@{-->}[ul] &					&	{#3}_{n-1}\ar@{-->}[ul]		&	{#3}_n\ar@{-->}[ul]	
}
}}

\newcommand{\ann}[1]{{\mathrm{ann}({#1})}}
\newcommand{\Stab}[1]{{\mathsf{Stab}\,{#1}}}

\newcommand{\Dom}[1]{{\mathrm{Dom}({#1}^{-1})}}

\newcommand{\mf}[1]{{\mathfrak{#1}}}

\newcommand{\bb}[1]{{\mathbb{#1}}}
\newcommand{\mca}[1]{{\mathcal{#1}}}
\newcommand{\mr}[1]{{\mathrm{#1}}}
\newcommand{\ms}[1]{{\mathscr{#1}}}
\newcommand{\mb}[1]{{\mathbf{#1}}}

\title{Stability conditions on affine Noetherian schemes}
\author{Kotaro Kawatani}
\email{kawatanikotaro@gmail.com}
\date{\today}


\begin{document}
\maketitle	

\begin{abstract}
We show that the existence of  locally finite stability conditions on the bounded derived category $\mathbf{D}^{b}(X)$ of coherent sheaves on an affine Noetherian scheme $X$ is equivalent to $\dim X=0$. 
We also study the space of stability conditions on the category of morphisms $\mathbf M_{X}$ in the derived category of the scheme $X$. 
%The category $\mb M_{X}$ is a generalization of the derived category of representations of $A_{2}$-quiver. 
Similarly to the case of $\mb D^{b}(X)$, the existence of $\Stab{\mb M_{X}}$ is equivalent to $\dim X=0$. 
%Moreover if $\dim X=0$ and $X$ is local, then $\Stab{\mb M_{X}}$ is isomorphic to the space of stability conditions on the derived categories of the $A_{2}$-quiver. 
Finally we show that the spaces of stability conditions on $\mathbf{D}^{b}(X)$ and on $\mathbf{M}_{X}$ are homotopy equivalent. 
\end{abstract}

\section{Introduction}

\subsection{Stability conditions on affine schemes}
Let $\mb D^{b}(X)$ be the bounded derived category of coherent sheaves on an algebraic variety $X$. 
The space $\Stab{\mb D^{b}(X)}$ of stability conditions on $\mb D^{b}(X)$, introduced by Bridgeland \cite{MR2373143}, 
is an effective mathematical object for the study of algebraic geometry. 
For instance $\Stab{\mb D^{b}(X)}$ has many applications not only to the derived category $\mb D^{b}(X)$ (cf. \cite{MR2376815} and \cite{MR3592689}) but also to moduli spaces of sheaves on the variety $X$ 
(cf. \cite{MR3010070}, \cite{MR3279532} and \cite{MR3194493}). 

Now the non-emptiness of $\Stab{\mb D^{b}(X)}$ is non-trivial and fundamental. 
If $X$ is smooth and projective with $\dim X \leq 2$, then $\Stab{\mb D^{b}(X)}$ is non-empty by \cite{MR2998828} and \cite{MR2373143}. 
If $X$ is smooth and projective with $\dim X=3$, %then the existence of stability conditions is equivalent to 
the generalized Bogomolov-Gieseker inequality introduced by Bayer-Macr\`{i}-Toda \cite{MR3121850} gives a sufficient condition for the non-emptiness.  
%Thus one could expect $\Stab{\mb D^{b}(X)}$ is non-empty if $X$ is smooth and projective. 
On the other hand, there seems no study of a necessary condition for the non-emptiness. 
In this note we show that the non-emptiness of stability conditions on an affine Noetherian scheme leads a property for dimension of the scheme. 

\begin{thm}\label{1st-thm}
Let $X$ be an affine Noetherian scheme of a Noetherian ring $R$. 
Then $\Stab{\mb D^{b}(X)}$ is non-empty if and only if $\dim X=0$. 
Moreover, if $\dim X=0$, then $\Stab{\mb D^{b}(X)}$ is isomorphic to $\bb C^{n}$ where $n$ is the number of the points in $X$. 
\end{thm}

Based on the theorem above, it might be interesting to study what kind of categorial properties of a triangulated category $\mb D$ guarantee the existence of (locally finite) stability conditions on $\mb D$. 
%Based on the theorem above, it might be interesting to study a relation between categorical property of a triangulated category $\mb D$ and the existence of locally finite stability conditions on $\mb D$. 
Roughly, our proof is based on the existence of ``many'' global functions on the affine scheme $X=\Spec R$. 
Though we do not understand such an existence in terms of triangulated categories, 
the following question represents one of the directions. 

\begin{question}
Let $\mb D$ be a $R$-linear triangulated category over a Noetherian ring $R$. 
Suppose $\Stab{\mb D} \neq \emptyset$. 
Does $\dim R=0$ holds?
\end{question}
%\begin{description}
%\item[Question] Let $\mb D$ be a $R$-linear triangulated category over a Noetherian ring $R$. 
%If $\Stab{\mb D}$ is non-empty, does the equality $\dim R=0$ holds?
%\end{description}
%%The question has to be solved with more sophistication, 
It might be interesting to study a relation between the existence of stability conditions and $R$-linear structures. 
To answer the question, a new idea could be needed.

%Based on the theorem above, it might be interesting to study a relation between geometric property of a Noetherian scheme $Y$ and the existence of locally finite stability conditions on $\mb D^{b}(Y)$. 
%Roughly, our proof is based on the existence of ``many'' global functions on the affine scheme $X$. 
%One might expect that the non-emptiness of $\Stab{\mb D^{b}(Y)}$ implies the existence of ``a few'' global functions. 
%The following question represents one of the directions. 
%
%\begin{description}
%\item[Question] If $\Stab{\mb D^{b}(Y)}$ is non-empty, is the scheme $Y$ proper over the base? 
%\end{description}
%%The question has to be solved with more sophistication, 
%To answer the question, a new idea could be needed. 

\subsection{Stability conditions on morphisms}

We further study the space of stability conditions on morphisms in the bounded derived category of an affine Noetherian scheme, which is our second interest. 
Note that the category of morphisms in a triangulated category is not triangulated in general. 
Recall that the derived category $\mb D^{b}(X)$ of a Noetherian scheme $X$ is obtained by the homotopy category $\mr{h}(\ms D^{b}_{\mr{coh}}(X))$ of a stable infinity category $\ms D^{b}_{\mr{coh}}(X)$ of quasi-coherent sheaves with bounded coherent cohomologies. 
Then the homotopy category $\mr{h}(\ms D^{b}_{\mr{coh}}(X)^{\Delta^{1}})$ of the infinity category $\ms D^{b}_{\mr{coh}}(X)^{\Delta^{1}}$ of morphisms in the infinity category $\ms D^{b}_{\mr{coh}}(X)$ is triangulated, and hence is a reasonable candidate of the triangulated category of morphisms in $\mb D^{b}(X)$.

From now on let us denote by $\mb M_{X}$ the category $\mr{h}(\ms D^{b}_{\mr{coh}}(X)^{\Delta^{1}})$ of morphisms. 
We note that $\mb M_{X}$ is equivalent to the bounded derived category of representations of the $A_{2}$ quiver when $X$ is the affine scheme of a field by the author \cite[Corollary 6.2]{morphismstability}.

Basically we are interested in a relation between the spaces of stability conditions on $\mb D^{b}(X)$ and on $\mb M_{X}$. 
In our previous paper \cite{morphismstability}, we showed that there exist natural two holomorphic maps $d_{0}^{*}, d_{1}^{*} \colon \Stab{\mb D^{b}(X)} \rightrightarrows \Stab{\mb M_{X}}$ and that these maps are closed embeddings whose images do not intersect each other. 
Moreover $\sigma \in \Stab{\mb D^{b}(X)}$ is full if and only if $d_{i}^{*}\sigma $ is full where $i \in \{0,1\}$. 
The following naive problem suggested in \cite{morphismstability} is based on an expectation of the contractibleness of the space of stability conditions: 
\begin{prob}[{\cite[Problem 1.1]{morphismstability}}]\label{problem1}
Is $\Stab{\mb M_{X}}$ homotopy equivalent to $\Stab{\mb D^{b}(X)}$?
\end{prob}

One of the natural expectations is that these spaces of stability conditions are both contractible unless they are empty, in particular they are homotopy equivalent to each other. 
However it seems difficult to prove the homotopy equivalence in general. 
Moreover, to the best of our knowledge, there are no examples of triangulated categories whose space of stability conditions are non-contractible. 
So if the answer to Problem \ref{problem1} is negative, we might find an interesting example which is non-contractible space. 

The second aim is to give an  answer to Problem \ref{problem1} when $X$ is an affine Noetherian scheme. 
More precisely we show the following: 
\begin{thm}[=Corollary \ref{cor:equivalent}]\label{thm2}
Let $X$ be an affine Noetherian scheme. 
%Both $\Stab{\mb M_{X}}$ and $\Stab{\mb D^{b}(X)}$ are contractible. 
$\Stab{\mb M_{X}}$ is homotopy equivalent to $\Stab{\mb D^{b}(X)}$. 
\end{thm}

The proof of Theorem \ref{thm2} might be interesting as follows. 
We first prove an analogous statement to Corollary \ref{cor:charcterization}: 
\begin{thm}[=Corollary \ref{cor:matome}, Theorem \ref{thm:main2}]\label{thm1}
Let $X$ be an affine Noetherian scheme.
Then $\Stab{\mb M_{X}}$ is non-empty if and only if $\dim X=0$. 
Moreover, if $\dim X=0$, then $\Stab{\mb M_{X}}$ is isomorphic to $\bb C^{2n} $ where $n$ is the number of points in $X$. 
%\begin{center}
%$(1)$ $\dim X$ is zero, $(2)$ $\Stab{\mb D^{b}(X)}$ is non-empty, and $(3)$ $\Stab{\mb M_{X}}$ is non-empty. 
%\end{center}
\end{thm}
Thus, if $\dim X> 0$ then there is nothing to study the homotopy types of $\Stab {\mb M_{X}}$ and of $\Stab {\mb D^{b}(X)}$. 
Now suppose $\dim X=0$. 
Roughly we show that the space of stability conditions is invariant under nilpotent thickening. 
To be precise, let $X_{0}$ be the set of closed points of $X$. 
Then the closed embedding $i \colon X_{0} \to X$ induces a faithful exact functor 
$i_{*} \colon \mb M_{X_{0}} \to \mb M_{X}$. 
In general an exact functor between triangulated categories does not induce a map between the spaces of stability conditions. 
However,  using the inducing construction  due to \cite{MR2524593}, we show that the functor $i_{*}$ implies a continuous map $i_{*}{}^{-1} \colon \Stab{\mb M_{X}} \to \Stab{\mb M_{X_{0}}}$ contravariantly and that the map gives an isomorphism. 
Using Qiu \cite{qiu-thesis} and Dimitrov-Katzarkov \cite{MR3984103}, we see that $\Stab{\mb M_{X_{0}}}$ is contractible. 
Then, combining Theorem \ref{1st-thm}, Theorem \ref{thm2} follows. 

So the essential part is the isomorphism $i_{*}{}^{-1} \colon \Stab{\mb M_{X}} \to \Stab{\mb M_{X_{0}}}$ induced by the faithful functor $i_{*}$. 
Due to non-functoriality of taking the space of stability conditions, our method might be rare and interesting.
\section{Preliminaries}

\subsection{Derived categories of coherent sheaves}
Let $X$ be a Noetherian scheme and $\mb D^{b}(X)$ be the bounded derived category of coherent sheaves on $X$. 
A global function $r \in H^{0}(X, \mca O_{X})$ gives an endomorphism 
$\mu_{r} \colon E \to E $ of $E \in \mb D^{b}(X)$ via multiplication by $r$. 
We refer to the morphism $\mu_{r}$ as the \textit{multiplication by $r \in H^{0}(X, \mca O_{X})$}. 
The multiplication $\mu_{r}$ is just the value of the morphism $\mu$ as algebras 
\begin{equation}\label{multi}
\mu \colon \Gamma (X, \mca O_{X}) \cong \Hom_{\mb D^{b}(X)}(\mca O_{X}, \mca O_{X}) \to \Hom_{\mb D^{b}(X)}(E, E). 
\end{equation}

The following condition for an object $E \in \mb D^{b}(X)$ is crucial for us: 
\begin{dfn}
An object $E$ in $\mb D^{b}(X)$ has the \textit{isomorphic property} if $E$ satisfies the following:
\begin{itemize}
\item[(Ism)] For any $r \in R$, the morphism $\mu_{r} \colon E \to E$ is an isomorphism if $\mu_{r}$ is non-zero.  
\end{itemize}

\end{dfn}

Recall that the \textit{support} $\Supp E$ of a complex $E \in \mb D^{b}(X)$ is the union of the support of the $i$-th cohomology of $E$: 
\[
\Supp E = \bigcup _{i \in \bb Z} \Supp  H^{i}(E). 
\]
Note that $\Supp E$ is closed since $E$ is a bounded complex.

\subsection{Inducing stability conditions}

Let $\mb D$ be a triangulated category. 
Following the original article \cite{MR2373143}, the set of locally finite stability conditions on $\mb D$ is denoted by $\Stab{\mb D}$. 
Recall that a stability condition consists of a pair $\sigma=(Z, \mca P)$ where $Z$ is a group homomorphism from the Grothendieck group of $\mb D$ to $\bb C$ and $\mca P=\{ \mca P(\phi) \}_{\phi \in \bb R}$ is the collection of full sub-abelian categories $\mca P(\phi)$ of $\mb D$. 
An object $A \in \mb D$ is said to be \textit{$\sigma$-semistable} if $A$ is in $\mca P(\phi)$ and $A$ is non-zero. 
Moreover the object $A$ is said to be \textit{$\sigma$-stable} if $A$ is simple in $\mca P(\phi)$.

%\begin{lem}\label{lem:summand}
%Let $\sigma$ be a locally finite stability condition on a triangulated category $\mb D$. 
%Suppose that the direct sum $A\oplus B \in \mb D$ is $\sigma$-semistable with phase $\phi_{0}$. 
%Then both $A$ and $B$ are $\sigma $-semistable with phase $\phi_{0}$. 
%\end{lem}
%
%\begin{proof}
%Let $p \colon A \+ B \to A$ be the projection and $i \colon A \to A \+ B$ the section of $p$. 
%Moreover let $C$ be a $\sigma$-semistable object with phase $\phi$. 
%We first show that $\Hom_{\mb D}(C, A)$ is zero if $\phi > \phi_{0}$. 
%
%Let $f$ be in $\Hom_{\mb D}(C, A)$. 
%Since $\phi > \phi_{0}$, we see the composite $i \circ f$ is zero. 
%Then $p \circ i \circ f = f$ is zero. 
%Thus we see $\Hom_{\mb D}(C, A)=0$ if $\phi > \phi_{0}$. 
%
%Similarly suppose $\phi < \phi_{0}$. 
%Let $g$ be in $\Hom_{\mb D}(A, C)$. 
%Then the composite $g \circ p \colon A \+ B \to C$ is zero since $\phi < \phi_{0}$. 
%Then $g = g \circ p \circ i$ has to be zero. 
%Thus $\Hom_{\mb D}(A, C)=0$ and $A$ has to be $\sigma$-semistable. 
%The assertion for $B$ is similar. 
%\end{proof}

An exact functor $F \colon \mb D \to \mb D'$ between triangulated categories does not induces a map between spaces of stability conditions in general. 
However, a ``good'' functor $F \colon \mb D \to \mb D'$ induces a map $F^{-1}$ from a subset of $\Stab{\mb D'}$ to $\Stab{\mb D}$ due to Macr\'{i}-Mehrotra-Stellari \cite{MR2524593}. 
Let us briefly recall the construction of $F^{-1}$. 

Let $F \colon \mb D \to \mb D'$ be an exact functor between triangulated categories. 
Assume that $F$ satisfies the following additional condition
\begin{enumerate}
\item[(Ind)] $\Hom_{\mb D'}(F(a) , F(b))=0$ implies $\Hom_{\mb D}(a,b)=0$ for any $a, b\in \mb D$. 
\end{enumerate}

Let $\sigma' =(Z', \mca P')\in \Stab {\mb D'}$. 
Define $F^{-1}\sigma'$ by the pair $(Z, \mca P)$ where 
\begin{equation}
Z = Z'\circ F, \ \mca P(\phi) = \{ x \in \mb D \mid F(x) \in \mca P'(\phi) \}. 
\end{equation}\label{eq:inducing}
By the definition of $F^{-1}\sigma'$, the pair $F^{-1} \sigma'$ is a stability condition on $\mb D$ if and only if $F^{-1} \sigma' $ has the Harder-Narasimhan property. 

\begin{lem}[{\cite[Lemma 2.9]{MR2524593}}]\label{lm:MMS}
Notation is the same as above. 
The map $F^{-1} \colon \Dom F \to \Stab {\mb D}$ is continuous. 
%\begin{enumerate}
%\item Then the set 
%\[
%\Dom F =\{	\sigma ' \in \Stab{ \mb D'} \mid F^{-1}	 \sigma'  \in \Stab {\mb D} \}
%\]
%is closed in $\Stab {\mb D'}$. 
%\item 
%
%\end{enumerate}
\end{lem}

\begin{rmk}
Recall that the universal cover $\widetilde{\mr {GL}}_{2}^{+}(\bb R)$ of $\mr{GL}_{2}^{+}(\bb R)$ has the right action to the space of stability conditions. 
The map $F^{-1}$ is $\widetilde{\mr {GL}}_{2}^{+}(\bb R)$-equivariant by the definition of $F^{-1}$. 
\end{rmk}

\subsection{Semiorthogonal decompositions and stability conditions}
Collins--Polishchuck \cite{MR2721656} proposed a construction of stability conditions on a triangulated category $\mb D$ from a semiorthogonal decomposition. % $\mb D= \<\mb D_1, \mb D_2  \>$. 
A key ingredient of the construction is a \textit{reasonable} stability condition on a triangulated category.   

\begin{dfn}[{\cite[pp. 568]{MR2721656}}]\label{dfn-reasonable}
A stability condition $\sigma =(\mca A, Z)$ on a triangulated category $\mb D$ is 
\textit{reasonable} if $\sigma$ satisfies
\[
0 < \inf	\{	|Z(E)| \in \bb R  \mid 	E\mbox{ is semistable in }\sigma  \}. 
\]
%The set of reasonable stability conditions on $\mb D$ is denoted by $\Stabr{\mb D}$. 
\end{dfn}

\begin{rmk}\label{rmk:reasonable}
A reasonable stability condition is locally finite by \cite[Lemma 1.1]{MR2721656}. 
Unfortunately we do not know whether the converse holds or not. 
For instance, 
if $\rank K_{0}(\mb D)=1$, then any stability condition on $\mb D$ is reasonable. 
\end{rmk}

Let $\mb D$ be a triangulated category. 
Recall that a pair $(\mb D_{1}, \mb D_{2})$ of full triangulated subcategories of $\mb D$ is said to be a \textit{semiorthogonal decomposition} of $\mb D$ if the pair satisfies 
\begin{itemize}
\item[(1)] $\Hom_{\mb D}(E_{2}, E_{1})=0$ for any $E_{i } \in \mb D_{i}$ ($i=1,2$), and 
\item[(2)] any object $E \in \mb D$ is decomposed into a pair of objects $E_{i}\in \mb D_{i}$ $(i=1,2)$ by the following distinguished triangle in $\mb D$:
\[
\xymatrix{
E_{2}	\ar[r]	&	E	\ar[r]	&	E_{1}	\ar[r]	&	E_{2}[1]. 
}
\]
\end{itemize}
The situation will be denoted by the symbol $\mb D=\< \mb D_{1}, \mb D_{2}\>$ or simply $\< \mb D_{1}, \mb D_{2}\>$. 
In addition to the first condition (1) above, if $\Hom_{\mb D}(E_{1}, E_{2})=0$ holds, the semiorthogonal decomposition is said to be \textit{orthogonal}.

\begin{prop}[{\cite{MR2721656}}]\label{CP2.2}
Let $\<\mb D_1, \mb D_2  \>$ be a semiorthogonal decomposition of a triangulated category $\mb D$. 
The left adjoint of the inclusion $\mb D_1 \to \mb D$ is denoted by $\tau_1$ and the right adjoint of the inclusion $\mb D_2 \to \mb D$ is denoted by $\tau_2$.
Let $\sigma_i = (Z_i, \mca P_i)$ be a reasonable stability condition on $\mb D_i$ for $i \in \{1,2\}$.  
Suppose that $\sigma_1$ and $\sigma_2$ satisfy the following conditions
 \begin{enumerate}
\item  $\Hom_{\mb D}^{\leq 0}\left (\mca P_1(0,1], \mca P _2(0,1]\right) =0$ \label{condition-a} and 
\item There is a real number $a\in (0,1)$ such that $\Hom_{\mb D}^{\leq 0} \left(\mca P_1(a,a+1], \mca P _2(a,a+1]\right) =0$. \label{condition-b}
\end{enumerate}

Then there exists a unique reasonable stability condition $\gl{\sigma_1} {\sigma_2}$ on $\mb D$ glued from $\sigma _1$ and $\sigma _2$ whose heart $\mca A$ of the $t$-structure of $\gl{\sigma _1} {\sigma _2}$ is given by 
\[
\mca A= \{ E \in \mb D \mid \tau_i(E) \in \mca P_i\left((0,1]\right)  \, (i=1,2)\}
\]
and whose central charge $Z$ is given by $Z(E) = Z_1(\tau_1(E)) + Z_2(\tau_2(E))$. 
\end{prop}

\subsection{A category of morphisms}

%%%%
Let $\mb D^{b}(X)$ be the bounded derived category of coherent sheaves on a Noetherian scheme $X$. 
The category of morphisms in $\mb D^{b}(X)$, introduced by the author, is one of generalizations of the derived category of representations of the $A_{2}$ quiver. 
Let us briefly recall the construction. 

Let $\ms D_{\mr{coh}}^{b}(X)$ be the stable infinity category of quasi-coherent sheaves on a Noetherian scheme $X$ with bounded coherent cohomologies. 
Then the homotopy category $\mr{h}(\ms D_{\mr{coh}}^{b}(X))$ of the infinity category is equivalent to the derived category $\mb D^{b}(X)$.

The homotopy category $\mr{h}(\ms D_{\mr{coh}}^{b}(X)^{\Delta^{1}})$ of the infinity category $\ms D_{\mr{coh}}^{b}(X)^{\Delta^{1}}$ of morphisms in the infinity category $\ms D_{\mr{coh}}^{b}(X)$ is a reasonable candidate of 
a triangulated category of morphisms in $\mb D^{b}(X)$. 
Thus we refer to $\mr{h}(\ms D_{\mr{coh}}^{b}(X)^{\Delta^{1}})$ as the \textit{category of morphisms in $\mb D^{b}(X)=\mr{h}(\ms D_{\mr{coh}}^{b}(X))$}. 

\begin{dfn}
Let $X$ be a Noetherian scheme. 
The category of morphisms in $X$ is denoted by $\mb M_{X}$. 
If $X$ is the affine scheme of a Noetherian ring $R$, we simply write $\mb M_{\Spec R}$ as $\mb M_{R}$. 
\end{dfn}

Note that a morphism $[f \colon E \to F]$ in $\mb D^{b}(X)$ determines an object in $\mb M_{X}$. 
There exist pairs of adjoint functors between $\mb D^{b}(X)$ and $\mb M_{X}$: 
\[
\xymatrix{
\mb D^{b}(X)	\ar[r]|-(.4)s&\ar@<-1.5ex>[l]|-{d_0}\ar@<1.5ex>[l]|-{d_1}	\mb M_{X}
}
; d_0 \dashv s \dashv d_1, 
\]
where $d_{0}([E\to F])=F$, $d_{1}([E \to F])=E$ and $s(E)=[\1_{E} \colon E \to E]$. 
Moreover $d_{1}$ has the right adjoint $j_{!}$ and $d_{0}$ has the left adjoint $j_{*}$
%Moreover there exists functors from $\mr{h}(\ms D_{\mr{coh}}^{b}(X)) \to \mr{h}(\ms D_{\mr{coh}}^{b}(X)^{\Delta^{1}})$; 
\[
j_{!} \colon \mb D^{b}(X) \to \mb M_{X}
%; j_{!}(x)=[x\to 0],
\mbox{ and  } 
j_{*} \colon \mb D^{b}(X) \to \mb M_{X}, %; j_{*}(x)=[0\to x]. 
\]
where $j_{!}(E)=[E \to 0]$ and $j_{*}(E)=[0 \to E]$.

\begin{lem}\label{lem:SOD}
Let $X$ be a Noetherian scheme and define the subcategories of the triangulated category $\mb M_{X}$ by 
\begin{align*}
(\mb M_{X})_{/0}&:=
\left\{
[E \to 0]
\middle|  E \in \mr{h}(\ms D_{\mr{coh}}^{b}(X))
\right\},	
\\
(\mb M_{X})_{0/}&:=
\left\{
[ 0 \to E]
\middle|  E \in \mr{h}(\ms D_{\mr{coh}}^{b}(X))
\right\}, \text{and}
\\
(\mb M_{X})_{s}&:=
\left\{
[\1  \colon E \to E]
\middle|  E \in \mr{h}(\ms D_{\mr{coh}}^{b}(X))
\right\}. 
\end{align*}
The triangulated category $\mb M_{X}$ has three semiorthogonal decompositions: 
\begin{align}
&\<	(\mb M_{X})_{s}, (\mb M_{X})_{/0}\> ,	\label{sd_{0}}
\\
&\<(\mb M_{X})_{0/}, (\mb M_{X})_{s}\>, \text{and} \label{d_{1}s}
\\
&	\<(\mb M_{X})_{/0}, (\mb M_{X})_{0/}\>. \label{nokori}
\end{align}
\end{lem}

\begin{proof}
The first two decompositions follow from \cite[Lemma 2.14]{morphismstability}. 
Since $j_{*}$ is the left adjoint of $d_{0}$ and $j_{!}$ is the right adjoint of $d_{1}$, we have canonical morphisms 
$j_{*} \circ d_{0} (f) \to f$ and $f \to j_{!} \circ d_{1}(f)$ for $f \in \mb M_{X}$. 
Then the sequence 
\[
\xymatrix{
j_{*}\circ  d_{0}(f)	\ar[r]	&	f	\ar[r]	&	j_{!} \circ d_{1} (f) 	\ar[r]	&	j_{*} \circ d_{0}(f)[1]
}
\]
gives a distinguished triangle in $\mb M_{X}$ since the triangulated structure on $\mb M_{X}$ is defined object-wise. 

Note that $(\mb M_{X})_{/0}$ (resp. $(\mb M_{X})_{0/}$) is the essential image of $j_{!}$ (resp. $j_{*}$). 
The adjunction $d_{1} \dashv j_{!} $ implies 
\[
\Hom_{\mb M_{X}}(j_{*}E, j_{!}F) \cong \Hom_{\mb D^{b}(X)}(d_{1}\circ j_{*}(E), F) =\Hom_{\mb D^{b}(X)}(0, F)=0. 
\]
This gives the proof of (\ref{nokori})
\end{proof}

\begin{rmk}
Let $\mb M_{X} = \< \mb M_{1}, \mb M_{2} \>	$ be one of  semiorthogonal decompositions in Lemma \ref{lem:SOD}. 
Then both components $\mb M_{1}$ and $\mb M_{2}$ are equivalent to $\mb D^{b}(X)$ in any cases and equivalences are respectively given by 
\begin{equation}\label{eq:identification}
\begin{cases}
s \colon \mb D^{b}(X) \to ( \mb M_{X})_{s}\\
j_{!} \colon \mb D^{b}(X) \to (\mb M_{X})_{/0}\\
j_{*} \colon \mb D^{b}(X) \to (\mb M_{X})_{0/}. 
\end{cases}
\end{equation}
Throughout this note, we always identify $\mb D^{b}(X) $ with components of semiorthogonal decompositions of $\mb M_{X}$. 

\end{rmk}

Let $(\mr{coh} \, X )^{\Delta^{1}}$ be the category of morphisms in the abelian category $\mr{coh}\, X$ of coherent sheaves on $X$. 
The category $(\mr{coh} \, X )^{\Delta^{1}}$ is also abelian, and 
we obtain the bounded derived category $\mb D^{b}\left((\mr{coh} \, X )^{\Delta^{1}}\right)$ by the localization of quasi-isomorphisms. 

\begin{prop}[{\cite[Corollary 6.2]{morphismstability}}]\label{prop:quiver}
Let $X$ be a Noetherian scheme. The triangulated category $\mb D^{b}\left((\mr{coh} \, X )^{\Delta^{1}}\right)$ is equivalent to $\mb M_{X}$. 
In particular the category $\mb M_{X}$ has a natural bounded $t$-structure $(\mb M_{X}^{\leq 0}, \mb M_{X}^{\geq 1})$ whose heart is equivalent to the abelian category $(\mr{coh}\, X)^{\Delta^{1}}$. 
\end{prop}

\begin{rmk}\label{rmk:quiver}
\begin{enumerate}
\item We refer to the $t$-structure $(\mb M_{X}^{\leq 0}, \mb M_{X}^{\geq 1})$ as the \textit{canonical $t$-structure on  $\mb M_{X}$}. 

\item If $X$ is $\Spec \mb k$ of a field $\mb k$ then $(\mr{coh}\, X)^{\Delta^{1}}$ is nothing but the abelian category of finite dimensional representations of the $A_{2}$ quiver. 

\end{enumerate}
\end{rmk}

%%%%%%%%%%%%%%%%%%%%%

\section{Stability conditions on affine schemes}

We study the space of stability condition on the derived category of an affine Noetherian scheme. 
%The main results of this section is Cor

\begin{lem}\label{lem:isom}
Let $R$ be a Noetherian domain with $\dim R >0$. 
Suppose that an $R$-module $M$ satisfies the following condition
\begin{itemize}
\item The morphism $\mu_{r} \colon M \to M$ is an isomorphism for any $r \in R \setminus \{0\}$. 
\end{itemize}
Then $M$ is zero. 
\end{lem}

\begin{proof}
Suppose to the contrary that $M \neq 0$. 
The assumption implies $\ann{M}=(0)$. 
Thus the support of $M$ is $\Spec R=X$. 

One can choose $r \in R =H^{0}(X, \mca O_{X})$ such that $r$ is not unit in $R$ since $\dim X>0$. 
Since the morphism $\mu_{r} \colon M \to M$ is an isomorphism, we have $M \otimes R/(r)=0$. 
Thus $\Supp M$ is a proper closed subset of $X$ and this gives a contradiction. 
Hence $M$ is zero.  
\end{proof}

\begin{lem}\label{lem:keylem}
Let $R$ be a Noetherian ring. 
Suppose that an $R$-module $M$ satisfies the condition (Ism). 
%\begin{itemize}
%\item[(Ism)] For an $r \in R$, the morphism $\mu_{r} \colon M \to M$ is an isomorphism if $\mu_{r}$ is non-zero.  
%\end{itemize}
Then the following holds. 
\begin{enumerate}
\item $\ann{M}$ is a prime ideal. 
\item If $M$ is non-zero, then $\ann{M}=\ann{m}$ for any $m \in M\setminus \{0\}$. 
In particular $\ann{M}$ is the unique associated prime of $M$. 
\item If $M$ is non-zero, then $\dim \Supp(M)=0$. 
\end{enumerate}
\end{lem}

\begin{proof}
Suppose that $ab\in \ann{M}$ and $a\not\in \ann{M}$. 
Then $\mu_{ab}=\mu_{a}\mu_{b}$ is the zero morphism. 
Since $\mu_{a}$ is an isomorphism by the condition (Ism), $\mu_{b}$ has to be zero and $b$ is in $\ann{M}$

Clearly we have $\ann{M} \subset \ann{m}$ for any $m \in M \setminus \{0\}$. 
Let $r$ be in $\ann{m}$. 
Then $\mu_{r}$ is not isomorphism. 
The condition (Ism) implies that $\mu_{r}$ is zero. 
Thus we see $\ann{M}=\ann{m}$. 
The last part of the second assertion is obvious.

To complete the proof, we show the assertion (3). 
Let $\mf p$ be the prime ideal $\ann{M}$. 
If $\mf p$ is not maximal, $M$ satisfies the condition (Ism) as $R/\mf p$-modules and we have $\dim R/\mf p>0$. 
%The condition (Ism) implies that any non-zero $\bar r \in R/\mf {p}$ gives an isomorphism $\mu_{\bar r} \colon M \to M$. 
Then Lemma \ref{lem:isom} implies $M=0$. 
Hence $\mf p$ has to be maximal and we see $\dim \Supp (M)=\dim \Supp (R/\mf p)=0$. 
\end{proof}

\begin{lem}\label{lem:zero}
Let $R$ be a Noetherian ring with $\dim R>0$. 
Suppose that an object $E \in \mb D^{b}(\Spec R)$ satisfies the condition $\mathrm{(Ism)}$. 
If $E$ is non-zero then $\dim \Supp (E) =0$. 
%: 
%\begin{itemize}
%\item[(Ism)] For an $f \in R$, the morphism $\mu_{f} \colon  E \to E$ is an isomorphism if $\mu_{f}$ is non-zero.  
%\end{itemize}
\end{lem}

\begin{proof}
Let $\mu_{r}^{i} \colon H^{i}(E) \to H^{i}(E)$ be the $i$-th cohomology of the morphism $\mu_{r} \colon E \to E$. 
Note that $\mu_{r}^{i}$ is also the multiplication by $r$ on $H^{i}(E)$. 
Since $E$ satisfies the condition (Ism), so does $H^{i}(E)$ if $H^{i}(E) \neq 0$. 
Thus Lemma \ref{lem:keylem} implies $\dim \Supp H^{i}(E)=0$ and we have the desired assertion. 
\end{proof}

\begin{thm}\label{thm:nonexistance}
Let $X$ be an affine Noetherian scheme with $\dim X>0$. 
%Suppose that $X$ is irreducible and has positive dimension. 
Then the set $\Stab{\mb D^{b}(X)}$ is empty. 
\end{thm}

\begin{proof}
We denote by $R$ the coordinated ring $H^{0}(X, \mca O_{X})$ of the affine scheme $X$. 
Suppose to the contrary that there exists a locally finite stability condition $\sigma \in \Stab {\mb D^{b}(X)}$.
Since $\sigma$ is locally finite, 
there exists a $\sigma$-stable object $A \in \mb D^{b}(X)$.

%Let us prove that the dimension of $\Supp A$ is zero if $A$ is non-zero. 
Note that any nonzero endomorphism $\varphi \colon A \to A$ is an isomorphism in $\mb D^{b}(X)$. 
Thus $A$ satisfies the condition (Ism) via the morphism (\ref{multi}). 
Hence we see $\dim \Supp(A)=0$ for any $\sigma$-stable object by Lemma \ref{lem:zero}. 

Taking the Harder-Narasimhan filtration and a Jordan-H\"{o}lder filtration, the structure sheaf $\mca O_{X}$ is given by a successive extension of finite $\sigma$-stable objects $\{A_{i}\}_{i=1}^{n}$. 
Thus we have $X= \bigcup_{i=1}^{n} \Supp A_{i}$ and this gives a contradiction since $\dim X\neq 0$. 
\end{proof}

Thus if $\dim X>0$ then there is nothing to study $\Stab{\mb D^{b}(X)}$. 
Next goal is to describe $\Stab{\mb D^{b}(X)}$ for the case $\dim X=0$. 

\begin{lem}\label{lem:negative}
Let $\mb D$ be a triangulated category, and $\mca A$ the heart of a bounded $t$-structure $(\mb D^{\leq 0}, \mb D^{\geq 1})$. 
The truncation functors with respect to the $t$-structure are respectively denoted by $\tau_{\leq 0} \colon \mb D \to \mb D ^{\leq  0}$ and $\tau_{\geq 1} \colon \mb  D \to \mb D^{\geq 1}$.  
The cohomology of $ E \in \mb D$ with respect to the $t$-structure is denoted by $H^{i}(E)$.

Suppose an object $E \in \mb D$ satisfies  
\[
\tau_{\geq m_{1}+1} E=0 \text{ and }\tau_{\leq m_{2}-1}E=0. 
%m_{1}= \max \{ i \in \bb Z \mid H^{i}(E)  \neq 0\}, \text{and }m_{2}=\min\{ i \in \bb Z \mid H^{i}(E) \neq 0\}. 
\]
Then $ \Hom_{\mb D}(E, E[m_{2}-m_{1}]) \cong \Hom_{\mb D}(H^{m_{1}}(E), H^{m_{2}}(E))$. 
\end{lem}

\begin{proof}
For the simplicity, we may assume $m_{2}=0$ by shifts, if necessary. 
%%一般の添字バージョンで書こうと思ったけれど, このほうが読みやすいよね...

The truncation functors give a distinguished triangle 
\[
\xymatrix{
(\tau_{\geq 1}E )[-1]\ar[r]	&	\tau_{\leq 0}E \ar[r]	&	E	\ar[r]	& \tau_{\geq 1} E	. 
}
\]
Since $E[m_{1}]$ belongs to $D^{\leq 0}$, we have 
$\Hom_{\mb D}(E[m_{1}], \tau_{\geq 1}E)=\Hom_{\mb D}(E[m_{1}], (\tau_{\geq 1}E)[-1])=0$. 
Thus we see 
\[
\Hom_{\mb D}(E[m_{1}], E) \cong \Hom_{\mb D}(E[m_{1}], \tau_{\leq 0}E ). 
\]
There is a distinguished triangle 
\[
\xymatrix{
\tau_{\leq -1}(E[m_{1}])	\ar[r]	&	E[m_{1}]	\ar[r]	&	\tau_{\geq 0}(E[m_{1}] )	\ar[r]	&	\tau_{\leq -1}(E[m_{1}])	[1]. 
}
\]
The assumption implies $\tau_{\leq 0}E \in \mb D^{\geq 0}$. 
Then the vanishings 
\[
\Hom_{\mb D}(\tau_{\leq -1}\left(E[m_{1}]),  \tau_{\leq 0}E\right) = \Hom_{\mb D}\left(\tau_{\leq -1}(E[m_{1}])[1],  \tau_{\leq 0}E\right) =0
\]
imply the isomorphism:  
\[
\Hom_{\mb D}(E[m_{1}], \tau_{\leq 0}E) \cong \Hom_{\mb D}\left(\tau_{\geq 0} (E[m_{1}]), \tau_{\leq 0}E \right). 
\]
Since $\tau_{\leq 0}E$ (resp. $\tau_{\geq 0}(E[m_{1}])$) is nothing but $H^{0}(E)$ (resp. $H^{m_{1}}(E)$), we obtain the desired assertion. 
\end{proof}

\begin{lem}
Let $R$ be a zero-dimensional Noetherian local ring. 
Then $\Stab{\mb D^{b}(\Spec R)}$ is non-empty.  
\end{lem}

\begin{proof}
Put $X=\Spec R$. 
Recall that any object in $\mr{coh}(X)$ is given by a successive extension of the residue field $R/ \mf m$. 
Hence one can define a group homomorphism $Z \colon K_{0}(\mb D^{b}(X)) \to \bb C$ by $Z(R/\mf m)=-1$. 
Then the pair $\sigma=(\mr{coh}(X), Z)$ has the Harder-Narasimhan property in the sense of \cite[Definition 2.3]{MR2373143} since any object  in $\mr{coh}(X)$ is $\sigma$-semistable. 
Thus $\sigma$ is a stability condition on $\mb D^{b}(X)$ by \cite[Proposition 5.3]{MR2373143}. 
The locally finiteness is obvious since the abelian category $\mr{coh}(X)$ is Artinian and Noetherian. 
Thus $\Stab {\mb D^{b}(\Spec R)}$ is not empty. 
\end{proof}

\begin{prop}\label{prop:zerolocal}
Let $R$ be a zero-dimensional Noetherian local ring. 
Then $\Stab{\mb D^{b}(\Spec R)}$ is isomorphic to $\bb C$. 
\end{prop}

\begin{proof}
Let us denote by $\mf m$ the maximal ideal of $R$. 
If once we show that $R/\mf m$ is stable for all stability conditions on $\mb D^{b}(\Spec R)$, 
the same argument in \cite[Proposition 3.7]{morphismstability} implies the desired assertion. 

%Let $A$ be a stable object for a stability condition. 
We claim that any stable object $A$ for a stability condition is a sheaf up to shifts. 
To show the claim set $m_{1}$ and $m_{2}$ by 
$m_{1}=\max\{i \in \bb Z \mid H^{i}(A) \neq 0\}$ and 
$m_{2}= \max \{	i	\in \bb Z \mid H^{i}(A) \neq 0\}$. 
It is enough to show that $m_{1}-m_{2}=0$. 

Next we have to show that $A$ is $R/ \mf m$ up to shifts. 
Lemma \ref{lem:zero} implies that the support of the stable object $A \in \mb D^{b}(\Spec R)$ is annihilated by the maximal ideal $\mf m$. 
Hence each cohomology of $A$ with respect to the standard $t$-structure is an $R/ \mf m$-module. 
Since $A$ is stable we have 
\[
\Hom(A, A[-m])=0. 
\]
for any positive integer $m$. 
%Recall the isomorphism $\Hom(A, A[m_{2}-m_{1}]) \cong \Hom(H^{m_{1}} (A), H^{m_{2}}(A)  )$ from Lemma \ref. 
Since the cohomologies of $A$ are vector spaces over the field $R/\mf m$, they are isomorphic to the direct sums of $R/\mf m$. 
Hence $m_{2}-m_{1}$ is zero by Lemma \ref{lem:negative}. 
Thus $A$ is isomorphic to the direct sum $(R/ \mf m)^{\oplus r}$ up to shifts. 
Since any non-zero endomorphism is invertible, 
$A$ is isomorphic to $R/\mf m$ up to shifts. 
\end{proof}

%\begin{rmk}
%We have an alternative proof of Proposition \ref{prop:zerolocal} by using the inducing construction. 
%
%\end{rmk}

\begin{thm}\label{thm:artin}
Let $X$ be an affine Noetherian scheme with $\dim X=0$. 
Then $\Stab{\mb D^{b}(X)}$ is isomorphic to $\bb C^{n}$ where $n $ is the number of points in $X$. 
%Let $R$ be a Noetherian ring with $\dim R=0$. 
%Then $\Stab{\mb D^{b}(\Spec R)}$ is not empty. 
\end{thm}

\begin{proof}
Let $R$ be the coordinate ring $H^{0}(X, \mca O_{X})$. 
%Note that the heart $\mr{coh}(X)$ of the standard $t$-structure on $\mb D^{b}(X)$ is Artinian and Noetherian. 
Recall that $R$ is the finite product of Noetherian local rings $\{R_{i} \}_{i=1}^{n}$ with $\dim R_{i}=0$.  
Then the derived category $\mb D^{b}(X)$ has the orthogonal decomposition 
\[
\mb D^{b}(X) = \bigoplus _{i=1}^{n} \mb D^{b}(\Spec R_{i}), 
\]
and the space $\Stab{\mb D^{b}(X)}$ is the finite product of $ \{ \Stab{ \mb D^{b}(\Spec R_{i})} \}_{i=1}^{n}$ by \cite[Proposition 5.2]{MR3984103}. 
Hence $\Stab {\mb D^{b}(\Spec R)}$ is isomorphic to $\bb C^{n}$ by Proposition \ref{prop:zerolocal}. 
\end{proof}

\begin{cor}\label{cor:charcterization}
Let $X$ be an affine Noetherian scheme. 
The space $\Stab{\mb D^{b}}(X)$ is not empty if and only if $\dim X=0$. 
\end{cor}

\begin{proof}
The proof is clear from Theorems \ref{thm:nonexistance} and \ref{thm:artin}. 
\end{proof}

\section{Stability conditions on morphisms}\label{sc:morphism}

Basically we are interested in a relation between $\Stab{\mb D^{b}(X)}$ and $\Stab{\mb M_{X}}$. 
One of motivated problems is Problem \ref{problem1}. 
We first study the non-emptiness of $\Stab {\mb M_{X}}$ for an affine Noetherian scheme $X$. 
%
%For instance if $X$ is $\Spec \mb k$ of a field $\mb k$, the answer is affirmative by \cite[Proposition 3.7]{morphismstability}. 
%Unfortunately it seems quite difficult to discuss other cases. 
%%From now on, we focus on the case $X$ is an affine Noetherian scheme. 

\begin{prop}\label{prop:morphisms}
Let $X$ be an affine Noetherian scheme. 
Then $\Stab{\mb M_{X}}$ is not empty if and only if $\dim X =0$. 
\end{prop}

\begin{proof}

Assume $\dim X=0$. 
Then $\Stab{\mb D^{b}(X)}$ is not empty by Proposition \ref{thm:artin}. 
So \cite[Theorem 1.2]{morphismstability} implies that $\Stab{\mb M_{X}}$ is not empty. 

Assume that $\dim X$ is positive. %Let $f $ be in $\mb D^{b}(X)^{\Delta^{1}}$. 
Take $f\in \mb M_{X}$ which is stable with respect to a stability condition on $\mb M_{X}$. 
Note that there is a morphism of algebras via term-wise multiplication: 
\[
\mu  \colon \Hom_{\mb D^{b}(X)}(\mca O_{X}, \mca O_{X}) \to \Hom_{\mb M_{X}}(f, f). 
\] 
In particular $d_{i} \mu \colon d_{i}f \to d_{i}f$ is also the multiplication. 
Since $f$ satisfies the condition (Ism), so does $d_{i}f  \in \mb D^{b}(X)$ ($i \in \{0, 1\}$). 
Lemma \ref{lem:zero} implies that the objects $d_{i}f$ supported in closed points of $X$. 
This gives a contradiction by the same reason in Theorem \ref{thm:nonexistance}. 
\end{proof}

\begin{cor}\label{cor:matome}
Let $X$ be an affine Noetherian scheme. 
The following are equivalent. 
\begin{enumerate}
\item The dimension of $X$ is zero, 
\item $\Stab{\mb D^{b}(X)}$ is non-empty, and 
\item $\Stab{\mb M_{X}}$ is non-empty. 
\end{enumerate}
\end{cor}

\begin{proof}
The proof is clear from Corollary \ref{cor:charcterization} and Proposition \ref{prop:morphisms}. 
\end{proof}

\begin{rmk}
If $\dim X$ is positive, then both $\Stab{\mb D^{b}(X)}$ and $\Stab{\mb M_{X}}$ are empty. 
In particular they are homotopy equivalent to each other. 
\end{rmk}

Now we further study $\Stab{\mb M_{X}}$ when $X$ is the affine scheme of a zero-dimensional Netherian local ring $R$. 
To simplify notation, we introduce the following: 

\begin{dfn}
Let $R$ be a zero-dimensional Noetherian local ring with maximal ideal $\mf m$ and let $\mb k$ be the residue field $R/ \mf m$. 
The categories of morphisms in $\mb D^{b}(\Spec \mb k)$ is denoted by 
$\mb M_{0}$. % := \mr{h}  (\ms D^{b}_{\mr{coh}} (\Spec R/\mf m)^{\Delta^{1}}). 
We denote by $\mca A_{R}$ (resp. $\mca A_{0}$) the heart of the standard $t$-structure $(\mb M_{R}^{\leq 0}, \mb M_{R}^{\geq 1})$ on $\mb M_{R}$ (resp. $(\mb M_{0}^{\leq 0}, \mb M_{0}^{\geq 1})$ on $\mb M_{0}$). 
\end{dfn}

The main theorem of this section is Theorem \ref{thm:main2} below. 
Let $i \colon \Spec \mb k \to \Spec R$ be the closed embedding. 
The exact functor $i_{*} \colon \mr{coh}(\Spec \mb k) \to \mr{coh}(\Spec R) $ induces a functor 
\[
\mca A_{0} \to \mca A_{R}
\]
which is also exact. 
Thus we obtain the functor
\[
 \mb M _{0}\to \mb M_{R}. 
\]
By abusing notation, we denote by $i_{*}$ these induced functors.

\begin{lem}\label{lem:ext1}
The functor $i_{*} \colon \mb M_{0} \to \mb M_{R}$ is faithful. 
\end{lem}

\begin{proof}
It is enough to show that the natural morphism 
\begin{equation}\label{eq:ext}
i_{*}^{p} \colon \Hom_{\mb M_{0}} (f, g[p])  \to \Hom_{\mb M_{R}}(i_{*}f, i_{*}g[p])
\end{equation}
is injective for any $f$ and $g \in \mb M_{0}$, and  for any integer $p$. 
Recall that any object in $\mb M_{0}$ is the finite direct sum of shifts of objects in $\mca A_{0}$. 
Hence we can assume both $f$ and $g$ are in $\mca A_{0}$ without loss of generality. 

Since the global dimension of $\mca A_{0}$ is $1$, the left hand side in (\ref{eq:ext}) vanishes for $p \not\in \{0,1\}$. 
Now the claim for $p=0$ is obvious since $i_{*}^{0}$ is an isomorphism. 
In addition $i_{*}^{1}$ is injective since $i_{*} \colon \mca A_{0} \to \mca A_{R}$ is exact and commutes with direct sums. 
\end{proof}

\begin{lem}\label{lem:higherext}
Let $f$ and $g$ be in $\mca A_{0}$. 
If $\Hom_{\mb M_{R}}(i_{*}f, i_{*}g[p])$ is zero for $p \in \{ 0,1\}$, then 
\[
\Hom_{\mb M_{R}}(i_{*}f, i_{*}g[p])=0
\] holds for any $p \in \bb Z$.  
\end{lem}

\begin{proof}
Recall that any object in $\mca A_{0}$ is the direct sum of indecomposable objects in $\mca A_{0}$ and any indecomposable object in $\mca A_{0}$ is one of the following: 
\begin{equation*}%\label{eq:indecomposable}
s(\mb k)=[\1 \colon \mb k \to \mb k], 
j_{!}(\mb k) = [\mb k \to 0], \text{or }
j_{*}(\mb k) =[0 \to \mb k]. 
\end{equation*}
It is enough to prove the claim when $f$ and $g$ are indecomposable.  

Lemma \ref{lem:ext1} implies $\Hom_{\mb M_{0}}(f, g[p])=0$ for $p \in \{0,1\}$. 
Then the pair $\left(i_{*}(f),i_{*}(g)\right)$ has to be one of the following:
\[
\left(i_{*}(f),i_{*}(g)\right)
=\left(j_{*}(\mb k), j_{!}(\mb k)\right), 
\left(s(\mb k), j_{*}(\mb k)\right), \text{or }
\left(j_{!}(\mb k), s(\mb k)\right). 
\]

Suppose $(i_{*}f, i_{*}g) = \left(j_{*}(\mb k), j_{!}(\mb k)\right)$. 
Recall that $j_{*}$ is the left adjoint of $d_{0}$. 
Hence we see 
\[
\Hom_{\mb M_{R}}(j_{*}(\mb k), j_{!}(\mb k)[p]) \cong 
\Hom_{\mb D^{b}(\Spec R)}(\mb k, d_{0} \circ j_{!}(\mb k)[p]) 
=
\Hom_{\mb D^{b}(\Spec R)}(\mb k, 0)=0. 
\]
Similarly one can prove the claim for $\left(i_{*}(f),i_{*}(g)\right)=\left(s(\mb k), j_{*}(\mb k)\right)$ by using the adjunction $s \dashv d_{1}$. 

Finally suppose that $\left(i_{*}(f),i_{*}(g)\right)=\left(j_{!}(\mb k), s(\mb k)\right)$. 
The adjunction $d_{0}\dashv s$ implies 
\[
\Hom_{\mb M_{R}}(j_{!}(\mb k), s(\mb k)[p]) 
\cong 
\Hom_{\mb D^{b}(\Spec R)}(d_{0} \circ j_{!}(\mb k), \mb k[p]) 
=
\Hom_{\mb D^{b}(\Spec R)}(0, \mb k[p]) =0. 
\]
We have finished the proof. 
\end{proof}

\begin{prop}\label{prop:classification}
Let $\sigma$ be a locally finite stability condition on $\mb M_{R}$. 
If $f \in \mb M_{R}$ is $\sigma$-stable, then $f$ is, up to shifts, one of the following: 
\begin{equation}\label{eq:indecomposable}
s(\mb k)=[\1 \colon \mb k \to \mb k], \ 
j_{!}(\mb k)=[\mb k\to 0], \text{ and }
j_{*}(\mb k)=[0 \to \mb k]. 
\end{equation}
\end{prop}

\begin{proof}
By the argument in Proposition \ref{prop:morphisms}, each $i$-th cohomology $H^{i}(f)$ of $f$ with respect to the $t$-structure $(\mb M_{R}^{\leq 0}, \mb M_{R}^{\geq 1})$ is in $\mca A_{0}$. 
It is enough to show that $f$ is in $\mca A_{0}$ up to shifts. 
In fact, if $f$ is in $\mca A_{0}$ and stable then $f$ is indecomposable by \cite[Lemma 3.3]{morphismstability}. 
Since there are only $3$ indecomposable objects in $\mca A_{0}$ listed in (\ref{eq:indecomposable}), we have the desired assertion. 

Without loss of generality, assume that $f$ satisfies 
\[
0 =\min \{ j \in \bb Z \mid H^{j}(f) \neq 0 \}. 
\]
Put $\ell = \max \{ j \in \bb Z \mid H^{j}(f)\neq 0 \}$ and it is enough to show that $\ell =0$.

Assume $\ell =1$. 
Now we claim $\Hom_{\mb M_{R}}(H^{1}(f), H^{0}(f)[1])=0$.  
Recall that there is a distinguished triangle 
\[
\xymatrix{
H^{0}(f) \ar[r]^-{\tau_{0}}	&	f	\ar[r]^-{\tau_{1}}	&	H^{1}(f)[-1]	\ar[r]	&	H^{0}(f) [1]
}
\]
by the truncation for the canonical $t$-structure. 
If $\Hom_{\mb M_{R}}(H^{1}(f), H^{0}(f)[1])\neq 0$, then there exists a non-zero morphism 
$\varphi \colon H^{1}(f)[-1] \to H^{0}(f)$. 
Then the composite $\tilde \varphi= \tau_{0} \circ \varphi \circ \tau_{1}$ is a non-zero endomorphism of $f$ since $f$ is in $\mb M_{R}^{\leq 1} \cap \mb M_{R}^{\geq 0}$. 
Then the $\sigma$-stability of $f$ implies that $\tilde \varphi$ is an isomorphism. 
On the other hand $\tilde \varphi \circ \tilde \varphi$ is zero by $\Hom_{\mb M_{R}} (H^{0}(f), H^{1}(f)[-1])=0$. 
Thus $\Hom_{\mb M_{R}}(H^{1}(f), H^{0}(f)[1])$ has to be zero if $\ell=1$. 

Recall the spectral sequence given by 
\begin{equation}\label{eq:spectral}
E_{2}^{p,q}=
\bigoplus _{i \in \mb Z} \Hom_{\mb M_{R}}(H^{i}(f), H^{i+q}(f)[p]) \Rightarrow 
\Hom_{\mb M_{R}}(f,f[p+q])=E^{p+q}. 
\end{equation}
Since we are assuming $\ell =1$, we see $E^{0,-1}_{2} \cong E^{0, -1}_{\infty}$.  
By the $\sigma$-stability of $f$, the vanishing $\Hom_{\mb M_{R}}(f, f[-1]) =0$ implies 
that $\Hom_{\mb M_{R}}(H^{1}(f), H^{0}(f))$ is zero. 
Then Lemma \ref{lem:higherext} implies the vanishings 
\[
\Hom_{\mb M_{R}}(H^{1}(f), H^{0}(f)[p])=0 \quad (\forall p \in \bb Z). 
\]
Thus $f$ has to be split. 
Since $f$ is indecomposable by \cite[Lemma 3.3]{morphismstability}, this gives a contradiction. 
Hence we see $\ell \neq 1$.

Next assume $\ell \geq 2$. 
By the spectral sequence (\ref{eq:spectral}), we see 
$E^{0, -\ell }_{2} \cong E^{0, -\ell }_{\infty}$ and $E^{1, -\ell }_{2} \cong E^{1, -\ell }_{\infty}$. 
Then the vanishing $\Hom_{\mb M_{R}}(f, f[-n])=0$ for any $n \in \bb N$ imply $E^{0, -\ell }_{2}=E^{1, -\ell}_{2}=0$. 
By Lemma \ref{lem:higherext}, we have
$\Hom_{\mb M_{R}}(H^{\ell }(f), H^{0}(f)[p])=0 $ for any $p \in \bb Z$. 
Thus 
$E^{p, -\ell}_{2}=0$  holds for any integer $p$. 
By induction on $q$, one can see the following 
\begin{equation}\label{eq:vanishing}
E^{0, q}_{2}=0 \text{ for }q<0 \text{ and }E^{1, q}_{2}=0 \text{ for }q<-1. 
\end{equation}

Then the following hold by (\ref{eq:vanishing}) : 
\begin{equation}\label{eq:vanishing2}
\begin{cases}
\Hom_{\mb M_{R}}(H^{\ell }(f), H^{0}(f) [q])=0	&	\forall q \in \bb Z	\\
\Hom_{\mb M_{R}}(H^{\ell }(f), H^{j}(f) )=0	&	0 < j < \ell	\\
\Hom_{\mb M_{R}}(H^{j}(f), H^{0}(f))=0	&	0 < j < \ell. 	\\
\end{cases}
\end{equation}
Since each cohomology is in the heart $\mca A_{0}$ and both $H^{0}(f)$ and $H^{\ell }(f)$ are non-zero, we see $H^{j}(f)=0$ for $0 < j < \ell$. 
Thus we obtain the distinguished triangle 
\[
\xymatrix{
H^{0}(f)	\ar[r]	&	f	\ar[r]	&	H^{\ell }(f)[-\ell]\ar[r]	&	H^{0}(f)[1]. 
}
\]
The first vanishing in (\ref{eq:vanishing2}) implies that $f$ is the direct sum $H^{0}(f) \+ H^{\ell }(f)[-\ell]$ and this gives a contradiction. 
Hence $\ell$ has to be zero. 
\end{proof}

\begin{thm}\label{thm:main2}
Let $i \colon \Spec \mb k \to \Spec R$ be the closed embedding. 
The following hold. 
\begin{enumerate}
\item $\Dom{i_{*}}=\Stab{\mb M_{R}}$
\item $i_{*}{}^{-1} \colon \Stab{\mb M_{R}} \to \Stab{\mb M_{0}}$ is an isomorphism as complex manifolds. 
\item $\Stab{\mb M_{R}}$ is isomorphic to $\bb C^{2}$. 
\end{enumerate}
\end{thm}

\begin{proof}
We first show $\Dom{i_{*}} \supset\Stab{\mb M_{R}}$. 
Take $\sigma \in \Stab{\mb M_{R}} $ arbitrary. 
It is enough to show that $i_{*}{}^{-1} \sigma$ defined by (\ref{eq:inducing}) has the Harder-Narasimhan property for any $f \in \mb M_{0}$. 

Note that an indecomposable object in $\mb M_{0}$ is also, up to shifts, one of the objects in (\ref{eq:indecomposable}) since the global dimension of $\mca A_{0}$ is $1$. 
If $f$ and $g$ in $\mb M_{0}$ has the Harder-Narasimhan filtration with respect to $i_{*}{}^{-1}\sigma$, 
one can construct the Harder-Narasimhan filtration of the direct sum $f \+ g$.  
So it is necessary to show that any indecomposable object $f \in \mb M_{0}$ has the Harder-Narasimhan filtration.

Since $\sigma$ is locally finite, any objet in $\mb M_{R}$ is given by a successive extension of finite $\sigma$-stable objects. 
Thus the classes of $\sigma$-stable objects generate $K_{0}(\mb M_{R}) \cong K_{0}(\mb D^{b}(\Spec R))^{\+2} $. 
Since $\rank K_{0}(\mb M_{R})=2$, two of the objects $f_{1}$ and $f_{2}$ in (\ref{eq:indecomposable}) should be $\sigma$-stable. 
If the other object $g$ in (\ref{eq:indecomposable}) is semistable then any indecomposable object in $\mb M_{0}$ has the trivial Harder-Narasimhan filtration for $i_{*}{}^{-1}\sigma$. 

We have to discuss three cases of $g$. 
Note that there is the following distinguished triangle of the objects in (\ref{eq:indecomposable}): 
\begin{equation}\label{eq:HN}
\xymatrix{
j_{*}(\mb k)	\ar[r]	&	s(\mb k)	\ar[r]		&	j_{!}(\mb k)	\ar[r]	&	j_{*}(\mb k)[1]
}
\end{equation}
There is no loss of generality in assuming that 
\begin{equation}\label{eq:kouho}
(f_{1}, f_{2})=
\begin{cases}
(s(\mb k), j_{!}(\mb k)[-1])	&	\text{if }g=j_{*}(\mb k)	\\
(j_{*}(\mb k)[1], s(\mb k))	&	\text{if }g=j_{!}(\mb k) \\
(j_{!}(\mb k), j_{*}(\mb k))	&	\text{if }g=s(\mb k)	. 
\end{cases}
\end{equation}
Then we have the distinguished triangle from the triangle (\ref{eq:HN}) in each cases:
\begin{equation}\label{eq:bunkai}
\xymatrix{
f_{2}	\ar[r]	&	g	\ar[r]	&	f_{1}	\ar[r]	&	f_{2}[1]. 
}
\end{equation}
Moreover the triplet $(f_{1}, f_{2}, g)$ corresponds to three semiorthogonal decomposition of $\mb M_{R}$ in Lemma \ref{lem:SOD}

Let $\phi _{i}$ be the phase of the stable object $f_{i}$. 
Since $\Hom_{\mb M_{R}}(f_{1}, f_{2}[1])$ is non-zero, the stability of $f_{1}$ and $f_{2}$ implies $\phi_{1} < \phi_{2}+1$. 
If $\phi _{2}> \phi_{1}$, then the filtration (\ref{eq:bunkai}) gives the Harder-Narasimhan filtration of $g$. 
%Suppose that $\phi_{1}\leq \phi_{2}$. 
If $\phi_{2}=\phi_{1}$, then $g=s(\mb k)$ is semistable by (\ref{eq:bunkai}). 
If $\phi_{2} < \phi_{1}$, then the inequality $\phi_{2} < \phi_{1} < \phi _{2}+1$ holds. 
Without loss of generality we can assume that the heart $\mca P(0,1]$ of $\sigma$ is the extension closure generated by $f_{1}$ and $f_{2}$ by $\widetilde{\mr{GL}}_{2}^{+}(\bb R)$-action.  
Then the non-trivial subobject of $g$ is only $f_{2}$. 
Hence $g$ is stable by $\phi_{2} < \phi_{1}$. 
Thus, if $g = s(\mb k)$, then $\sigma$ is in $\Dom{i_{*}}$ and we have $\Dom{i_{*}}=\Stab{\mb M_{R}}$.

Since $\Dom{i_{*}}=\Stab{\mb M_{R}}$, the map $i_{*}{}^{-1}\colon \Stab{\mb M_{R}} \to \Stab{\mb M_{0}} $ is not only continuous but also holomorphic. 
Thus it is enough to show that $i_{*}{}^{-1}$ is bijective since the spaces are complex manifolds. 

For the subjectivity, let $\sigma_{0}=(Z_{0}, \mca P_{0})$ be in $\Stab{\mb M_{0}}$. 
Then two of the objects listed in (\ref{eq:HN}) has to be $\sigma_{0}$-stable. 
Hence there are three possibilities of two objects $(f_{1}, f_{2})$. 
Without loss of generality we can assume that the pair $(f_{1}, f_{2})$ is (\ref{eq:kouho}). 
Note that these pairs generate semiorthogonal decompositions of not only of $\mb M_{0}$ but also of $\mb M_{R}$ listed in (\ref{sd_{0}}), (\ref{d_{1}s}) and (\ref{nokori}) respectively. 

%\[
%(f_{2}, f_{1})= (j_{*}(\mb k), j_{!}(\mb k)), (s(\mb k), j_{*}(\mb k)[1]), \text{ or }(j_{!}(\mb k), s(\mb k)[1])
%\]  
Then, in any cases, the pair satisfies 
\[
\begin{cases}
\Hom_{\mb M_{R}}(f_{2}, f_{1}[p])=0 & (\forall p)	\\
\Hom_{\mb M_{R}}(f_{1}, f_{2}[p])=0 & (p \leq 0)	\\
\Hom_{\mb M_{R}}(f_{1}, f_{2}[1])\neq 0. 
\end{cases}
\]

Let $n$ be the minimal integer which is greater than or equal to $\phi_{2}-\phi_{1}$. 
Using the identification (\ref{eq:identification}), define stability conditions $\sigma_{i}=(Z_{i}, \mca P_{i})$ on $\mb D^{b}(\Spec R) $ by 
\begin{align*}
\mca P_{1}(0,1]&=\mr{coh}(\Spec R), Z_{1}(\mb k):=Z_{0}(f_{1})	, \text{ and}\\
\mca P_{2}(0,1] &= \mr{coh}(\Spec R)[-n], Z_{2}(\mb k):=Z_{0}(f_{2}). 
\end{align*}
Note that both $\sigma_{1}$ and $\sigma_{2}$ are reasonable by Remark \ref{rmk:reasonable}. 
Since the set of phases of semistable objects for $\sigma_{i}$ is discrete, 
the second condition in Lemma \ref{CP2.2} is automatic. 
Then the gluing stability condition $\sigma:=\gl{\sigma_{1}}{\sigma_{2}}$ with respect to the corresponding  semiorthogonal decomposition on $\mb M_{R}$ satisfies $i_{*}{}^{-1}\sigma=\sigma_{0}$. 
Hence $i_{*}{}^{-1}$ is surjective. 

For the infectivity, let $\tau_{1}$ and $\tau_{2} $ be stability conditions on $\mb M_{R}$. 
If $i_{*}{}^{-1}(\tau_{1})=i_{*}{}^{-1}(\tau_{2}) =:\tau_{0}$, then there exist two indecomposable objects $f_{1}$ and $f_{2}$ in $\mb M_{0}$ which are stable in $\tau_{0}$ and whose phases are in the interval $(0,1]$. 
Since the heart of $\tau_{i}$ is the extension closure of $f_{1}$ and $f_{2}$, we see $\tau_{1}=\tau_{2}$.

Finally the third assertion follows from \cite{qiu-thesis}. 
In fact, $\Stab{\mb M_{0}}$ is isomorphic to $\bb C^{2}$ by \cite{qiu-thesis}. 
Thus the isomorphism $ i_{*}{}^{-1} \colon   \Stab{\mb M_{R}}  \overset{\sim}{\to}\Stab{\mb M_{0}}$ implies the assertion. 
\end{proof}

\begin{rmk}
The argument above gives an alternative proof of Proposition \ref{prop:zerolocal} as follows. 

Since the global dimension of $\mb k$ is zero, the functor $i_{*}\colon \mb D^{b}(\Spec \mb k) \to \mb D^{b}(\Spec R) $ is faithful. 
By the same argument in the proof of Theorem \ref{thm:main2}, we see $\Stab{\mb D^{b}(\Spec R)} = \Dom{i_{*}}$. 
Thus we obtain a holomorphic map 
\[
i_{*}{}^{-1} \colon \Stab{\mb D^{b}(\Spec R)} \to \Stab{\mb D^{b}(\Spec \mb k)}. 
\]
Then one can easily see that $i_{*}{}^{-1}$ is surjective and injective. 
Thus $i_{*}{}^{-1}$ gives an isomorphism between $\Stab {\mb D^{b}(\Spec R)}$ and $\Stab{\mb D^{b}(\Spec \mb k)}$. 
\end{rmk}

\begin{cor}\label{cor:tukareta}
Let $R$ be a zero-dimensional Noetherian ring. 
The spaces $\Stab{\mb M_{R}}$ is isomorphic to $\bb C^{2n}$ where $n$ is the number of points of $\Spec R$. 
\end{cor}

\begin{proof}
By the assumption, $R$ is the direct product $\prod_{i=1}^{n} R_{i}$ of zero-dimensional Noetherian local ring $\{R_{i}\}_{i=1}^{n}$. 
Then $\mb M_{R}$ has the orthogonal decomposition $\mb M_{R} = \bigoplus \mb M_{R_{i}}$ and $\Stab{\mb M_{R}}$ is isomorphic to the product 
$\prod _{i=1}^{n} \Stab{\mb M_{R_{i}}}$ by \cite[Proposition 5.2]{MR3984103}. 
Since each $\Stab{\mb M_{R_{i}}}$ is isomorphic to $\bb C^{2}$, we have the desired assertion. 
\end{proof}

\begin{cor}\label{cor:equivalent}
Let $R$ be an Noetherian ring. 
The spaces $\Stab{\mb M_{R}}$ and $\Stab{\mb D^{b}(\Spec R)}$ are homotopy equivalent. 
\end{cor}

\begin{proof}
Suppose $\dim R >0$. 
By Corollary \ref{cor:matome}, both are empty sets. 

Suppose $\dim R=0$.  
Let $n$ be the number of points in $\Spec R$. 
Then $\Stab{\mb D^{b}(\Spec R)}$ is isomorphic to $\bb C^{n}$ by Theorem \ref{thm:artin}. 
In particular $\Stab{\mb D^{b}(\Spec R)}$ is contractible. 
By Corollary \ref{cor:tukareta}, $\Stab{\mb M_{R}}$ is isomorphic to $\bb C^{2n}$. 
Hence $\Stab{\mb M_{R}}$ is also contractible and we have finished the proof. 
\end{proof}

%
%\bibliographystyle{plain}
%\bibliography{/Users/kotaro/Dropbox/MyPaper/bibs-kk}
%

\end{document}